\documentclass{amsart}
\usepackage{amsmath, amssymb, amsthm}

\newcommand{\nc}{\newcommand}
\nc{\w}{\omega}
\nc{\s}{\sigma}
\nc{\cl}[1]{\overline{#1}}
\nc{\sm}{\setminus}
\nc{\sbst}{\subseteq}

\newtheorem{thm}{Theorem}[section]
\newtheorem{lem}[thm]{Lemma}
\newtheorem{clm}[thm]{Claim}
\newtheorem{cor}[thm]{Corollary}

\newtheorem{qu}[thm]{Question}

\theoremstyle{definition}

\theoremstyle{remark}

\newcommand{\uhr}{\upharpoonright}

\newcommand{\nothing}[1]{}

\begin{document}

\title[Sequential properties of function spaces]{Sequential properties of function spaces with the compact-open topology}

\author{Gary Gruenhage}
\address[Gruenhage]{Department of Mathematics and Statistics, Auburn University, Auburn, AL 36830, USA}
\email{garyg@auburn.edu}

\author{Boaz Tsaban}
\address[Tsaban]{Department of Mathematics, Bar-Ilan University, Ramat-Gan 52900, Israel}
\email{tsaban@math.biu.ac.il}

\author{Lyubomyr Zdomskyy}
\address[Zdomskyy]{Kurt G\"odel Research Center for Mathematical Logic, University of Vienna,
W\"ahringer Str.\ 25, 1090 Vienna, Austria}
\email{lzdomsky@logic.univie.ac.at}


\begin{abstract}
Let $M$ be the countably infinite metric fan.   We show that
$C_k(M,2)$ is sequential and contains a closed copy of Arens space $S_2$.  It follows
that if $X$ is metrizable but not locally compact, then $C_k(X)$ contains a closed copy
of $S_2$, and hence does not have the property AP.

We also show that, for any zero-dimensional Polish space $X$,  $C_k(X,2)$
is sequential if and only if $X$ is either locally compact or the derived set $X'$ is compact.
In the case that $X$ is a non-locally compact Polish space whose derived set is compact,
we show that all spaces $C_k(X, 2)$
are homeomorphic, having the  topology determined by an increasing  sequence of Cantor subspaces, the $n$th one
nowhere dense in the $(n+1)$st.
\end{abstract}

\maketitle

\section{Introduction}

Let $C_k(X)$ be the space of continuous real-valued functions on $X$ with the
compact-open topology.  $C_k(X)$ for metrizable $X$ is typically not a $k$-space,
in particular not sequential.
 Indeed, by a theorem of R. Pol \cite{pol},  for $X$
paracompact first countable (in particular, metrizable),
$C_k(X)$ is a $k$-space if and only if $X$ is locally compact, in which case $X$ is a topological sum of locally compact $\s$-compact spaces
 and  $C_k(X)$ is a product of completely metrizable spaces.
 A similar result holds for $C_k(X,[0,1])$:
it is a $k$-space if and only if $X$ is the topological sum of
 a discrete space and a locally compact $\s$-compact space, in which case $C_k(X)$
is the product of a compact
 space and a completely metrizable space.
It follows that, for separable metric $X$, the following are equivalent:
\begin{enumerate}
\item $C_k(X)$ is a $k$-space;
\item $C_k(X)$ is first countable;
\item $C_k(X)$ is a complete separable metrizable space, i.e., a Polish space;
\item $X$ is a locally compact Polish space.
\end{enumerate}
The same equivalences hold for $C_k(X, [0,1])$.
On the other hand, for Polish $X$, $C_k(X)$ always has the (strong) Pytkeev
property \cite{tz}.

A space $X$ has the property \emph{AP} if whenever $x\in \cl{A}\sm
A$, there is some $B\sbst A$ such that $x\in \cl{B}\sbst
A\cup\{x\}$.   $X$ has the property \emph{WAP} when a subset $A$
of $X$ is closed if and only if there is no $B\sbst A$ such that $|\cl{B}\sm
A|=1$. Thus, every Fr\'echet space is AP and every sequential
space is WAP.  It was asked in \cite{gt} whether $C_k(\w^\w)$ is
WAP.

In this note, we first show that if $X$ is metrizable but not
locally compact, then $C_k(X)$ contains a closed copy of Arens space
$S_2$, and hence is not AP.  In fact, such a closed copy of $S_2$ is
contained in $C_k(M,2)$, where $M$ is the countable metric fan.  We
then show that $C_k(M,2)$ is sequential, in contrast to the full function space $C_k(M)$.
 Next we show that for a zero-dimensional Polish
space $X$, if $C_k(X,2)$ is not metrizable (which is the case if and only if
$X$ is not locally compact), then $C_k(X,2)$ is sequential if and only if the
derived set $X'$ is compact. We obtain a complete description of $C_k(X,2)$ for a
non-locally compact Polish $X$ such that $X'$ is compact: any such $C_k(X,2)$
is homeomorphic to the space $(2^\w)^{\infty}$, which is the space
with the  topology determined by an increasing sequence of
Cantor sets, the $n$th one nowhere dense in the $(n+1)$st.

\section{When $C_k(X)$ contains $S_2$}

\emph{Arens's space} $S_2$ is the set
$$\{(0,0),(1/n,0),(1/n,1/nm):n,m\in\w\setminus\{0\}\}\sbst
\mathbb{R}^2$$
carrying the strongest topology inducing the
original planar topology on the convergent sequences
$C_0=\{(0,0),(\frac1n,0):n >0\}$ and $C_n=\{(\frac1n,0),
(\frac1n,\frac1{nm}):m>0\}$, $n>0$. The \emph{sequential fan} is the
quotient space $S_\omega=S_2/C_0$ obtained from the Arens space
by identifying the points of the sequence $C_0$ \cite{Lin}.
$S_\omega$ is a non-metrizable Fr\'echet-Urysohn space, and $S_2$ is sequential and not Fr\'echet-Urysohn.
In fact, any space which is sequential but not Fr\'echet-Urysohn contains $S_2$ as a subspace.

The \emph{countably infinite metric fan} is the space
$M = (\w\times \w) \cup \{\infty\}$, where points of $\w\times \w$ are isolated, and the
basic neighborhoods of $\infty$ are $U(n) = \{\infty\}\cup ((\w\setminus n)\times \w)$, $n\in \w$.
$M$ is not locally compact at its non-isolated point $\infty$. 

\begin{lem}\label{l2_1}
$C_k(M,2)$ contains a closed copy of $S_2$.
\end{lem}
\begin{proof}

For each $n>1$ and each $k$, let
$$U(n,k) =( \{0\}\times n)\  \cup\  ((n\setminus \{0\}) \times k)\
\cup\ U(n),$$
and let $f_{n,k}$ be the member of $C_k(M)$ which is $0$ on
$U(n,k)$ and $1$ otherwise (i.e., the characteristic function of
$M\setminus U(n,k)$).

Let $f_n\in C_k(M)$ be the function which is $1$ on $\{0\}\times (\w\setminus n)$
and $0$ otherwise, and let $c_0$ be the constant $0$ function.
For each $n$, $\lim_k f_{n,k}=f_n$, and $\lim_n f_n=c_0$. Thus, $c_0$ is a limit point of the set $A=\{f_{n,k}: n>1,k\in  \w\}$.
Let $S=\{f_n:n>1\}$, and $X=\{c_0\}\cup S\cup A$.

We claim that $X$ is homeomorphic to the Arens space $S_2$.
It suffices to show that for each sequence $(k_n)_{n>1}$,
$c_0$ is not in the closure of the set $\{f_{n,k}:k<k_n, \ n>1\}$.
Given $(k_n)_{n>1}$, set
$K=\{(n-1,k_n):n>1\}\cup\{\infty\}$. Then $K$ is a sequence convergent to
$\infty$, and for each $f_{n,k}\in \{f_{n,k}:k<k_n, \ n>1\}$
there exists $x\in K$, namely, $x=(n-1,k_n)$, such that $f(x)=1$.
Therefore ${\{f_{n,k}:k<k_n, \ n>1\}}$ does not intersect the neighborhood
$\{f\in C_k(M,2): f\uhr K\equiv 0\}$ of $c_0$, and hence does not
contain $c_0$ in its closure.

By \cite[Corollary 2.6]{Lin}, if every point $z$ in a topological space $Z$
is regular $G_\delta$
(i.e., $\{z\}$  is equal to $\bigcap_n \cl{U_n}$ for some open neighborhoods $U_n$
of $z$),
and $Z$ contains a copy of $S_2$, then $Z$
contains a closed copy of $S_2$. Since every point of $C_k(M,2)$
is regular $G_\delta$, $C_k(M,2)$ contains a closed copy of
$S_2$. In fact, the space $X$ constructed above is closed, even in $C_p(M,2)$.
\end{proof}

\begin{thm}  \label{t2_1}
If $X$ is metrizable and not locally compact, then $C_k(X)$ contains
closed copies of $S_2$ and $S_\w$. \qed
\end{thm}

\begin{proof}

By Lemma 8.3 of \cite{vD}, a first
countable space $X$ contains a closed topological copy of the
space $M$ if and only if $X$ is not locally compact.
E.A. Michael \cite[Theorem 7.1]{Mi} observed that, for $Y$ a closed subspace
of a metrizable space $X$,  the linear extender $e:C(Y)\to C(X)$
given by the Dugundji extension theorem is a homeomorphic embedding when both $C(Y)$ and $C(X)$
are given the compact-open
topology (or the topology of uniform convergence, or pointwise convergence).
Thus we have that for each metrizable space $X$ which is not locally compact,
$C_k(M,2)$ is closely embedded in $C_k(X)$, and hence $C_k(X)$ contains a closed copy
of $S_2$.   Finally, $C_k(X)$ also contains a closed copy of $S_\w$ because for any topological group $G$,
$G$ contains a (closed) copy of $S_2$ if and only if it contains a closed copy of $S_\w$.
\end{proof}

{\bf Remark.} C.J.R. Borges \cite{Bo} showed that the Dugundji extension theorem holds for
the class of stratifiable spaces, and hence Theorem~\ref{t2_1} holds more generally for first
countable stratifiable spaces.

\section{ Sequentiality of $C_k(X,2
)$} \label{sequent}

A topological space $X$ \emph{carries the inductive topology}
with respect to a closed cover $\mathcal{C}$ of $X$, if for each $F\sbst X$, $F$ is closed
whenever $F\cap C$ is closed in $C$ for each $C\in \mathcal{C}$.
A topological space is a \emph{$k$-space} (respectively, \emph{sequential space})
if it carries the inductive topology with respect to its cover by compact (respectively, compact metrizable) subspaces.
$X$ is sequential if and only if for every non-closed  $A\sbst X$, there exists a sequence in $A$ converging to a point in $X\setminus A$.

Since the metric fan $M$ is not locally compact, $C_k(M)$ and $C_k(M,[0,1])$ are not
$k$-spaces \cite{pol}.  However, we have the following.

\begin{thm}  \label{t3_1}
$C_k(M,2)$ is sequential.
 \end{thm}
\begin{proof}
 Suppose not. Then there is $A\sbst C_k(M,2)$
which is not closed and yet contains
all limit points of convergent sequences of its elements.
As $M$ is zero-dimensional, $C_k(M, 2)$ is homogeneous. Thus, without loss of generality,
we may assume that $c_0 \in\cl{A}\setminus A$, where $c_0$ is the constant $0$ function.
We may additionally assume that $f (\infty) = 0$  for all $f \in A$.
    Let    $ A_n = \{ f \in  A : f (U (n)) = \{0\}\}$.


Note that the sets $A_n$ are increasing with $n$, and their union is $A$.
\begin{clm} \label{clma}
There exists a sequence $(k_n)_{n\in\w}$  such that for each $n$ with $f \in A_{n+1}$,
$1 \in f (\bigcup_{i\leq n} \{i\} \times k_i)$.
\end{clm}
\begin{proof}
By induction. Assume that for all $i < n$, there are $k_i$ such that $f \in A_{i+1}$ implies
$1 \in f (\bigcup_{j\leq i} \{j\} \times k_j)$, but that for each $k$, there is
$f_k \in A_{n+1}$ such that $f_k((\bigcup_{i<n} \{i\}\times k_i)\cup(\{n\}\times k)) =
 \{0\}$.
Let $f_k' = f_k\uhr (n+1)\times\w$.
As $2^{(n+1)\times\w}$ is homeomorphic to the Cantor space,
there is a subsequence $\{f'_{k_i}\}$ of $\{f'_k\}$, converging
to an element $f' \in 2^{(n+1)\times\w}$.
As $f_k \in A_{n+1}$, $f_k(U(n + 1)) = \{0\}$.
Define $g\in C_k(M)$ by $g(U (n + 1)) = \{0\}$ and $g \uhr (n + 1) \times\w = f'$.
Then in $C_k(M)$, $g=\lim_if_{k_i}$, and therefore $g \in A$.
As $f'_k(\{n\} \times k) = \{0\}$, $g(\{n\} \times\w) = \{0\}$.
As $g(U (n + 1)) = \{0\}$, $g(U (n)) = \{0\}$, and thus $g \in A_n$.
But $g(\bigcup_{i\leq n-1} \{i\} \times k_i) = \{0\}$
(indeed, this holds for all $f_k$'s), contradicting
the induction hypothesis.
\end{proof}

Let
$$K=(\bigcup_{i\in\w}\{i\} \times k_i )\cup \{\infty\}.$$
Let $V$ be the set of all functions which map $K$ into the interval $(-1/2, 1/2)$. 
Then $V$ is a neighborhood  of $c_0$ which misses $A$, a contradiction.
\end{proof}

We proceed to characterize the zero-dimensional Polish spaces $X$ such that $C_k(X,2)$ is sequential.

A topological space $Y$ has the \emph{strong Pytkeev property}
\cite{tz} (respectively, \emph{countable $\mathrm{cs}^\ast$-character})
if for each $y\in Y$, there is a \emph{countable} family $\mathcal N$
of subsets of $Y$, such that for each neighborhood $U$ of $y$ and
each $A\sbst Y$ with $y\in\cl{A}\sm A$ (respectively, each sequence $A$ in $Y\setminus\{y\}$ converging to $y$),
there is $N\in\mathcal N$ such that $N\sbst U$ and $N\cap A$ is infinite.

For every Polish space $X$ the space $C_k(X)$ has the strong
Pytkeev property \cite[Corollary~8]{tz}. Thus, any subspace of
$C_k(X)$ has the strong Pytkeev property, and therefore has countable $\mathrm{cs}^\ast$-character.

An \emph{$mk_\omega$-space} is a topological space which carries the inductive topology with
respect to a countable cover of compact metrizable subspaces.
A topological group $G$ is an \emph{$mk_\omega$-group} if $G$ is an $mk_\omega$-space.

$C_k(X,2)$ has a natural structure of a topological group.

\begin{thm}[\cite{bz}] \label{mainbz}
Let $G$ be a sequential non-metrizable topological group with countable
$\mathrm{cs}^*$-character.  Then $G$  contains an open
$mk_\omega$-subgroup $H$ and thus is homeomorphic
to the product $H\times D$ for some discrete space $D$.
\end{thm}

\begin{cor} \label{c1}
Let $G$ be a  sequential separable  topological group with
countable $\mathrm{cs}^\ast$-character. If $G$ is not metrizable,
then $G$ is $\sigma$-compact.
\end{cor}

\begin{lem} \label{metr}
Let $X$ be a zero-dimentional first countable  space. Then $C_k(X,2)$
is metrizable if and only if $X$ is locally compact and $\sigma$-compact.
\end{lem}
\begin{proof}
$(\Leftarrow)$
As $C_k(X,2)$ is a topological group, its metrizability
is equivalent to its first countability at $c_0$, the
constant zero function.

$(\Rightarrow)$ Assume that $C_k(X,2)$ is metrizable and fix a countable base
$\{W_n:n\in\w\}$ at $c_0$. Without loss of generality,
$W_n=\{f\in C_k(X,2):f\uhr K_n\equiv 0\}$
for some compact $K_n\sbst X$, and $K_n\sbst K_{n+1}$ for all $n$.
It suffices to prove that for every $x\in X$ there are a neighborhood
$U$ of $x$ and $n\in\w$, such that $U\sbst K_n$. If not, we can find $x\in X$
and a sequence $(x_n)_{n\in\w}$ of elements of $X$ such that $x_n\in U_n\setminus
K_n$, where $\{U_n:n\in\w\}$ is a decreasing base at $x$.
Set $K=\{x\}\cup\{x_n:n\in\w\}$ and $W=\{f\in C_k(X,2):f\uhr K\equiv 0\}$.
Since $K_n\cap K$ is finite for every $n\in\w$, there exists a function
$f\in C_k(X,2)$ such that $f\uhr K_n\equiv 0$ but $f\uhr K\not\equiv 0$, and hence
$W_n\not\sbst W$ for all $n\in\w$. This contradicts our assumption that
$\{W_n\}$ is a local base at $c_0$.
\end{proof}

For a topological space $X$, $X'$ is the set of all non-isolated points of $X$.

\begin{thm} \label{t3_2}
Let $X$ be a zero-dimensional Polish space which is not locally compact.
Then $C_k(X,2)$ is sequential if and only if the derived set $X'$ is compact.
\end{thm}
\begin{proof}
Assume that $X'$ is compact and consider the subgroup
$H=\{f\in C_k(X,2):f\uhr X'\equiv 0\}$.  $H$ is an open subgroup of
$C_k(X,2)$, and thus it suffices to prove that $H$ is
sequential. Since $X$ is not locally compact,  there is a clopen
outer base $\{U_n:n\in\w\}$ of $X'$ such that $U_0=X$,
$U_{n+1}\sbst U_n$, and $U_n\setminus U_{n+1}$ is infinite for
all $n\in\w$. Let $f:X\to M$ be a map such that $f(X')=\{\infty\}$
and $f\uhr(U_n\setminus U_{n+1})$ is an injective map onto $
\{n\}\times\w$. Then the map
$$f^\ast: \{g\in C_k(M,2):g(\infty)=0\}\to H $$
assigning to $g$ the composition $g\circ f$ is easily seen to be a
homeomorphism, and hence $H$ is sequential.
\medskip

Now assume that $X'$ is not compact.
  Then  there exists a  countable
closed discrete subspace $T\sbst X'$, and hence there exists a
discrete family $\{U_t:t\in T\}$ of clopen  subsets of $X$ such
that $t\in U_t$ for all $t\in T$.  $C_k(X,2)$ contains a
closed copy of the product $\Pi_{t\in T} C_k(U_t,2)$.

\begin{clm} \label{l2}
Let $Z$ be a non-discrete metrizable separable zero-dimensional
space. Then $C_k(Z,2)$ is not compact.
 \end{clm}
\begin{proof}
If $Z$ is locally compact, then it contains a clopen infinite
compact subset $C$. Then $C_k(C,2)$ is a closed subset of
$C_k(Z,2)$ homeomorphic to $\w$, and hence $C_k(Z,2)$
is not compact.

If $Z$ is not locally compact, then $Z$ contains a closed copy $Y$
of $M$. By Lemma \ref{l2_1}, $C_k(Y,2)$ contains a closed copy of $S_2$,
and is thus not compact. As restriction to $Y$ is a continuous map
from $C_k(Z,2)$ onto $C_k(Y,2)$, $C_k(Z,2)$ is not compact.
\end{proof}

\begin{clm} \label{cl1}
If none of the spaces $X_i$, $i\in\w$, is compact, then the
product $\Pi_{i\in\w}X_i$ is not $\sigma$-compact.
\end{clm}
\begin{proof}
A simple diagonalization argument.
\end{proof}

Since $T\sbst X'$, $U_t$ is not discrete for all $t\in T$.
By Claims~\ref{cl1} and \ref{l2}, the product $\Pi_{t\in T} C_k(U_t,2)$ is not $\sigma$-compact.
Thus, $C_k(X,2)$ is not $\sigma$-compact.

As $X$ is Polish, $C_k(X)$ has the strong Pytkeev property \cite{tz},
and thus has countable $\mathrm{cs}^*$-character. Consequently, so does
its subspace $C_k(X,2)$.
As $C_k(X,2)$ is separable and $X$ is not locally compact, $C_k(X,2)$ is not first countable,
and hence it is not metrizable.
Apply Corollary~\ref{c1}.
\end{proof}

\begin{cor}
$C_k(\w\times M, 2)$ is not sequential.\qed
\end{cor}

Let $(0)\in 2^\w$ be the constant zero sequence. Following \cite{ban}, let $(2^\w)^\infty$ be the space $\bigcup_{n\in\w}(2^{\w})^n$,
where $(2^{\w})^n$ is identified
with the subspace  $(2^{\w})^n\times \{(0)\}$ of $(2^{\w})^{n+1}$,
with the inductive topology with respect to the cover
$\{(2^\w)^n: n\in\w\}$.

\begin{thm}[Banakh \cite{ban}]\label{ban_t}
Every non-metrizable uncountable zero-dimensional
${mk}_\omega$-group is homeomorphic to $(2^\w)^{\infty}$.
\end{thm}

\begin{cor} \label{main_cor}
For zero-dimensional Polish spaces $X$, the following are equivalent:
\begin{enumerate}
\item $C_k(X,2)$ is sequential but not metrizable;
\item $C_k(X,2)$ is homeomorphic to $(2^\w)^\infty$;
\item $X$ is not locally compact
but $X'$ is compact.
\end{enumerate}
\end{cor}
\begin{proof}
$(1)\to (3)$. Since $C_k(X,2)$ is not metrizable, $X$ is not locally compact
(Lemma~\ref{metr}). By Theorem~\ref{t3_2}, $X'$ is compact.

$(3)\to(1)$. Since $X$ is not locally compact, $C_k(X,2)$ is not metrizable
(Lemma~\ref{metr}). By Theorem~\ref{t3_2}, the compactness of $X'$ implies
that $C_k(X,2)$ is sequential.

$(1)\to(2)$.
By \cite[Corollary~8]{tz}, $C_k(X,2)$ has countable $\mathrm{cs}^*$-character.
Applying Theorem~\ref{mainbz}, we have that $C_k(X,2)$ contains
an open $mk_\w$-subgroup. Since $C_k(X,2)$
is separable, it is an $mk_\w$-group.
Apply Theorem~\ref{ban_t}.
\end{proof}

$C_k(\w\times M,2)$ is homeomorphic to $C_k(M,2)^\w$, and hence to $((2^\w)^\infty)^\w$.
Thus, a negative answer to the following question would imply that
$C_k(P)$ is not WAP for ``most'' Polish spaces, including $\w^\w$ and some $\sigma$-compact
ones.

\begin{qu}
Does the space  $((2^\w)^\infty)^\w$ have the WAP property?
What about $S_2^\w$?
\end{qu}


\begin{thebibliography}{3}

\bibitem{ban}
T. Banakh,
\emph{Topological classification of zero-dimensional ${\mathcal M}_\omega$-groups},
Matematychni Studii \textbf{15} (2001), 109-112.

\bibitem{Bo} C.J.R. Borges, \emph{On stratifiable spaces}, Pacific Journal of  Mathematics
\textbf{17}  (1966), 1--16.

\bibitem{bz}
T. Banakh and L. Zdomskyy,
\emph{The topological structure of (homogeneous) spaces and groups with countable $\rm cs^*$-character},
Applied General Topology \textbf{5} (2004),  25--48.

\bibitem{vD}
E. van Douwen,
\emph{The integers and Topology}, in: K.Kunen, J.E.Vaughan (eds.), \textbf{Handbook of Set-Theoretic Topology},
North-Holland, Amsterdam, 1984, 111--167.



\bibitem{gt} G. Gruenhage and K. Tamano,
\emph{If $X$ is $\sigma$-compact Polish, then $C_k(X)$ has a $\sigma$-closure-preserving base},
Topology and its Applications \textbf{151}  (2005), 99--106.

\bibitem{Lin} S. Lin,
\emph{A note on Arens space and sequential fan},
Topology and its Appl.ications \textbf{81} (1997), 185--196.

\bibitem{Mi}E. A. Michael, \emph{Some extension theorems for continuous functions},
Pacific Journal of  Mathematics  \textbf{3} (1953), 789-806.



\bibitem{pol} R. Pol,
\emph{Normality in function spaces},
Fundamenta Mathematicae \textbf{84} (1974), 145--155.

\bibitem{tz}
B. Tsaban and L. Zdomskyy,
\emph{On the Pytkeev property in spaces of continuous functions, II},
Houston Journal of Mathematics \textbf{35} (2009), 563--571.

\end{thebibliography}
\end{document}